\documentclass{amsart}

\usepackage{amsmath}
\usepackage{amssymb}
\usepackage{MnSymbol}
\usepackage{dsfont}
\usepackage{epsfig}  		
\usepackage{xypic}
\usepackage{bbm}
\usepackage{hyperref}
\usepackage{eucal}
\usepackage{mathrsfs}

\usepackage{xcolor}

\newtheorem{mthm}{Theorem}

\newtheorem{mcor}[mthm]{Corollary}

\newtheorem{thm}{Theorem}[section]
\newtheorem{prop}[thm]{Proposition}
\newtheorem{lem}[thm]{Lemma}
\newtheorem{cor}[thm]{Corollary}

\theoremstyle{definition}
\newtheorem{definition}[thm]{Definition}

\theoremstyle{remark}

\newtheorem{remark}[thm]{Remark}

\numberwithin{equation}{section}

\newcommand{\CC}{\mathds{C}}
\newcommand{\RR}{\mathds{R}}

\newcommand{\PP}{\mathds{P}} 

\newcommand{\id}{\mathrm{id}}
\newcommand{\ad}{\mathrm{ad}}
\newcommand{\Ad}{\mathrm{Ad}}

\newcommand{\Iso}{\group{Iso}}

\newcommand{\GL}{\group{GL}}
\newcommand{\PGL}{\group{PGL}}
\newcommand{\SL}{\group{SL}}

\newcommand{\SO}{\group{SO}}

\newcommand{\Uni}{\mathrm{U}}
\newcommand{\uni}{\mathrm{u}}

\newcommand{\group}{\mathrm} 

\newcommand{\ag}[1]{\boldsymbol{#1}} 
\newcommand{\ra}[1]{\text{\textsf{\itshape #1}}} 

\newcommand{\aH}{\ag{H}}
\newcommand{\aU}{\ag{U}}
\newcommand{\aT}{\ag{T}}
\newcommand{\aG}{\ag{G}}

\newcommand{\raA}{\ra{A}}

\newcommand{\raT}{\ra{T}}
\newcommand{\raG}{\ra{G}}

\newcommand{\ac}[1]{\overline{#1}^{\rm z}}

\renewcommand{\rho}{\varrho}
\renewcommand{\tilde}{\widetilde}
\renewcommand{\bar}{\overline}

\renewcommand{\epsilon}{\varepsilon}

\newcommand{\Z}{\mathrm{Z}}

\newcommand{\zsp}{\mathbf{0}} 
\newcommand{\met}{\langle\cdot,\cdot\rangle}
\newcommand{\one}{\{e\}}

\newcommand{\spl}{\mathrm{s}}

\DeclareMathOperator{\im}{\mathrm{im}}

\newcommand{\frg}{\mathfrak{g}}
\newcommand{\frh}{\mathfrak{h}}
\newcommand{\fra}{\mathfrak{a}}
\newcommand{\frb}{\mathfrak{b}}

\newcommand{\frn}{\mathfrak{n}}
\newcommand{\frz}{\mathfrak{z}}

\newcommand{\frw}{\mathfrak{w}}

\newcommand{\frj}{\mathfrak{j}}

\newcommand{\fru}{\mathfrak{u}}
\newcommand{\frs}{\mathfrak{s}}
\newcommand{\frl}{\mathfrak{l}}

\newcommand{\frr}{\mathfrak{r}}
\newcommand{\frv}{\mathfrak{v}}

\newcommand{\zen}{\mathfrak{z}}

\newcommand{\inv}{\mathrm{inv}}

\newcommand{\p}{\mathsf{p}}

\newcommand{\g}{\mathrm{g}}
\newcommand{\I}{\mathrm{I}}

\newcommand{\wolf}[1]{}
\newcommand{\oliver}[1]{}

\begin{document}

\title[Rigidity of compact pseudo-Riemannian homogeneous spaces]{Rigidity of compact pseudo-Riemannian homogeneous spaces for solvable Lie groups}

\author[Baues]{Oliver Baues}
\address{Oliver Baues, Mathematisches Institut,
Georg-August-Universit\"at G\"ottingen,
Bunsenstr.~3-5, 
37073 G\"ottingen, Germany}
\email{obaues@uni-math.gwdg.de}

\author[Globke]{Wolfgang Globke}
\address{Wolfgang Globke, School of Mathematical Sciences, The University of Adelaide, SA 5005, Australia}
\email{wolfgang.globke@adelaide.edu.au}



\begin{abstract}
Let $M$ be a compact connected pseudo-Riemannian manifold
on which a solvable connected Lie group $G$ of isometries acts
transitively.
We show that $G$ acts almost freely on $M$ and 
that the metric on $M$ is induced by a bi-invariant
pseudo-Riemannian metric on $G$.
Furthermore, we show that the identity
component of the isometry group of $M$ coincides with $G$.
\end{abstract}

\maketitle

%
%

\tableofcontents


\section{Introduction and main results}

%

As exemplified by D'Ambra and Gromov's programmatic survey
\cite{DAG},
there has been a considerable interest in transformation groups
of manifolds with rigid geometric structures, among which
pseudo-Riemannian metrics, along with conformal and affine
structures, feature prominently.
In this context, isometry groups are typically assumed to be
non-compact in order to allow for sufficiently rich geometric and
dynamical properties,
whereas the manifolds are compact to ensure the geometries are
``almost classifiable'' in the words of \cite{DAG}.

Beside the Riemannian case, the \emph{Lorentzian} manifolds
(of metric signature $1$) constitute the most prominent class
of pseudo-Riemannian manifolds.
Zimmer \cite{zimmer4} studied semisimple Lie groups acting on
compact Lorentzian manifolds.
Adams and Stuck \cite{AS1} and Zeghib \cite{zeghib}
independently refined Zimmer's results into a classification
of the isometry groups of compact Lorentzian manifolds.
The case of higher signature pseudo-Riemannian metrics seems
much more difficult.

In this context, the most fundamental geometric objects are
\emph{homogeneous} manifolds, that is,
those admitting a transitive action by a group of isometries.
A classification of compact Lorentzian homogeneous spaces
was given by Zeghib \cite{zeghib}.
In a recent article, Quiroga-Barranco \cite{quiroga}
investigated transitive simple Lie groups
of isometries on compact pseudo-Riemannian manifolds
of arbitrary signature.
In the present article, we study transitive isometric
actions of \emph{solvable} Lie groups.

\subsection{The main results}

Let $M$ be a compact pseudo-Riemannian manifold,
and let $G$ be a connected solvable Lie group of isometries acting
transitively on $M$.

\begin{mthm}\label{mthm_lattice}
$G$ acts almost freely on $M$.
\end{mthm}

Theorem \ref{mthm_lattice} states that the stabilizer $\Gamma = G_x$ of any point $x\in M$ is a discrete subgroup in $G$.
Therefore, the orbit map
\[
o_{x}:  G  \to M ,\quad g \mapsto g \cdot x
\]
is a covering map. Since $o_{x}$ is a local diffeomorphism, 
the pseudo-Riemannian metric $\g$ on $M$ pulls back to a left-invariant non-degenerate metric tensor, and thus defines a pseudo-Riemannian metric $\g_G$ on $G$.
By construction, $\g_G$ is also invariant under conjugation by
$\Gamma$. This  subgroup is uniform in $G$ since $M$ is compact. 
We prove that the invariance under the uniform subgroup $\Gamma$ extends to all of $G$:

\begin{mthm}\label{mthm_invariant}
Let $\g_G$ be
the pulled-back left-invariant pseudo-Riemannian metric on $G$ as above. Then $\g_G$ is a bi-invariant pseudo-Riemannian metric.
\end{mthm}

Here,  a left-invariant metric $\g_G$ on $G$ is called \emph{bi-invariant} if the right multi\-plication map
$G \to G$, $h \mapsto h g$ is an isometry for all $g\in G$.\\


The above two theorems exhibit strong restrictions on
transitive isometric actions which are
imposed by the pseudo-Riemannian structure.
As Johnson \cite{johnson} showed,
every compact homogeneous space for a solvable Lie group
(except the circle) admits transitive solvable actions of
arbi\-trarily large dimensions.
Therefore, such actions cannot preserve a pseudo-Riemannian metric.
In addition, uniform subgroups in simply connected solvable Lie
groups are not always Zariski-dense in the adjoint representation,
so there is no apparent reason for a $\Gamma$-invariant metric
to be bi-invariant.
Such types of lattices appear already in the Lorentzian case
(see Medina and Revoy \cite{MR2}).

Let us further remark that, contras\-ting Theorems
\ref{mthm_lattice} and \ref{mthm_invariant},
Zwart and Boothby \cite[Section 7]{ZB} constructed 
transitive solvable actions with non-discrete stabilizer on compact symplectic manifolds which do not pull back to bi-invariant 
skew forms.
\\

Theorems \ref{mthm_lattice} and \ref{mthm_invariant} partially
generalize the results of Zeghib \cite[Th\'eor\`eme 1.7]{zeghib} on compact
Lorentzian homogeneous spaces with non-compact isometry groups.

Another special case are \emph{flat} compact
pseudo-Riemannian homogeneous manifolds.
It was noted in Baues \cite[Chapter 4]{baues2} that these
are precisely the quotients of two-step nilpotent Lie groups with
bi-invariant pseudo-Riemannian metrics by lattice subgroups.\\


%
Since every Lie group with bi-invariant metric is
a symmetric space (O'Neill \cite[Chapter 11]{oneill}), 
we obtain:

\begin{mcor}\label{mcor_symmetric}
The universal cover of $M$ is a
pseudo-Riemannian symmetric space.
In particular, $M$ is a locally symmetric space.
\end{mcor}

%

Recall that a manifold is called \emph{aspherical} if its
universal covering space is contractible. In particular,
every homogeneous space $M$ for a solvable Lie group is aspherical.
Such $M$ are often referred to as \emph{solvmanifolds}.
For comparison, note that any simple Lie group that acts
on a compact homogeneous aspherical manifold
is locally isomorphic to $\SL_2(\RR)$.
Note also that $\SL_2(\RR)$ can act locally effectively
on compact solvmanifolds, for example on the two-torus.

\begin{mcor}\label{mcor_isoM}
Let $M$ be a compact aspherical homogeneous pseudo-Riemannian
manifold with solvable fundamental group.
Then the connected component of\/ $\Iso(M)$ is solvable and acts
almost freely on $M$.
\end{mcor}

Corollary \ref{mcor_isoM} can be viewed as a
consequence of Gromov's Centralizer Theorem,
which implies that no group locally isomorphic to $\SL_2(\RR)$
can act on a compact analytic manifold with solvable fundamental
group (compare Gromov \cite[0.7.A]{gromov}).
Instead, we base our proof of Corollary \ref{mcor_isoM}
on the more general Theorem \ref{thm:invarvolume} below,
which concerns measure preserving transitive
actions on aspherical manifolds.

Moreover,
Corollary \ref{mcor_isoM} shows that in the homogeneous case\footnote{Results by An \cite{an} also indicate a relation
in the non-homogeneous case.}
the fundamental group
determines the structure of the isometry group to a large extent.
Indeed, a simply connected solvable Lie group is determined by
a lattice up to a compact deformation, see Baues and Klopsch
\cite{BK} (compare also Theorem \ref{thm:density} below).\\

We turn now to the problem of determining the isometry types
with given fundamental group:
Let $G$ be a simply connected  Lie group, $\g_{G}$ a bi-invariant 
pseudo-Riemannian metric on $G$, and $\Gamma \leq G$ a lattice. 
This turns $G/\,\Gamma$ into a pseudo-Riemannian manifold with 
metric inherited from $\g_{G}$, and $G$ acts on $G/\,\Gamma$ by 
isometries.  
A set of data $\mathcal{P}_{M} = (G, \g_{G}, \Gamma, \phi)$, where \[
\phi: G/\,\Gamma\,\to\, M
\]
is an isometry, is called a \emph{presentation for $M$} by the Lie group $G$ with bi-invariant metric $\g_G$. 
Let  $\mathcal{P}_{i} = (G_i, \g_{G_{i}}, \Gamma_i, \phi_{i})$ 
be presentations for $M_{1}$, $M_{2}$ respectively.  
An \emph{isometry of presentations}
\[
\Psi: \mathcal{P}_{1}  \to \mathcal{P}_{2}
\]
is an isomorphism of Lie groups $ \Psi: G_{1} \to G_{2}$,  which satisfies  \begin{enumerate} \item $\Psi(\Gamma_{1}) = \Gamma_{2}$. 
\item
$\Psi$ is an isometry with respect to the metrics 
$\g_{G_{1}}$ and $\g_{G_{2}}$. 
\end{enumerate} 
In particular, $\Psi$ defines induced isometries 
of pseudo-Riemannian manifolds
\[
\bar{\Psi}: G_1/\,\Gamma_{1} \to G_2/\,\Gamma_{2}
\quad \text{ and } \quad
\psi = \phi_{2} \bar{\Psi}  \phi_{1}^{-1} : M_{1} \to M_{2} \; .
\] 
By Theorem A and Theorem B, every compact pseudo-Riemannian
manifold $M$ with solvable isometry group has a presentation 
$\mathcal{P}$ by a Lie group with bi-invariant metric.
With these preliminaries in place 
we can show in Section \ref{sec_rigidity}: 

\begin{mcor}\label{mcor_rigidity}
Let $M_1$ and $M_2$ be compact pseudo-Riemannian manifolds with presentations $\mathcal{P}_{1}$ and 
$\mathcal{P}_{2}$ by Lie groups with bi-invariant metrics. Let 
$x_{i} = \phi_{i}(e \, \Gamma_{i})$ be the base points. 
Then every isometry  $\psi: M_{1} \to M_{2}$ with $\psi(x_{1}) = x_{2}$ 
is induced by an isometry of presentations 
$ \Psi: \mathcal{P}_{1}  \to \mathcal{P}_{2}$.
In particular, any two presentations of $M$ by Lie groups with bi-invariant metric are isometric. 
\end{mcor}

Corollary \ref{mcor_rigidity} provides us with an effective
procedure to classify compact homo\-geneous pseudo-Riemannian
manifolds $M$ with a transitive solvable isometry group,
by classifying simply connected Lie groups with bi-invariant metrics
and their lattices up to equivalence under Lie group
automorphisms.


\subsection{Further results and applications}


The proofs of Theorems \ref{mthm_lattice} and \ref{mthm_invariant},
given in Section \ref{sec_proofs},
rest on a careful
analysis of the symmetric bilinear form $\met$ induced by
$\g_G$ on the Lie algebra $\frg$ of $G$.
A priori, $\met$ is $\Ad(\Gamma)$-invariant, so
by continuity it is also invariant under the Zariski closure
$\ac{\Ad(\Gamma)}$.
However, in general a uniform subgroup of a
solvable Lie group $G$ is not Zariski-dense in $G$.\footnote{Baues and Klopsch exhibit
examples of lattices which are not Zariski-dense in
\cite[Chapter 2]{BK}.}
The analogous situation for semisimple Lie groups
is by comparison well understood through Borel's density
theorem \cite{borel0},
which states that a lattice in a semisimple Lie group $S$ without
compact factors is Zariski-dense in any linear representation
of $S$.
For solvable Lie groups there is a collection of density results
in special cases, see for example
Malcev \cite{malcev},
Baues and Klopsch \cite[Lemma 3.5]{BK},
Raghunathan \cite[Theorem 3.2]{raghunathan}
or Saito \cite[Th\'eor\`eme 3]{saito}.
These special cases are subsumed in the following density theorem:

\begin{thm}\label{thm:density}
Let $G$ be a connected solvable Lie group and $H$ a
uniform subgroup,
and let $\rho:G\to\GL(V)$ be a representation on a
finite-dimensional real vector space $V$. 
Let $\ra{A}$ denote the Zariski closure of $\rho(G)$ in $\GL(V)$. 
Then
\[
\ac{\rho(H)} \supseteq \raA_\spl,
\]
where $\raA_\spl$ is the maximal trigonalizable subgroup of $\raA$.
\end{thm}

Applied in the context of pseudo-Riemannian solvmanifolds,
this density theorem implies the following property:
The scalar product $\met$ induced by $\g_G$ on the Lie algebra
$\frg$ is \emph{nil-invariant}.
This means if
$\ad(X)_{\rm n}$ is the nilpotent part of the Jordan decomposition
of $\ad(X)$ for $X\in\frg$,
then $\ad(X)_{\rm n}$ is a skew operator with respect to $\met$.
In Sections \ref{sec_h_0} and \ref{sec_reduction} we study
the properties of nil-invariant scalar products.
The main result is:

\begin{thm}\label{thm_nilinvariant}
Let $\frg$ be a solvable Lie algebra and  $\met$ a 
nil-invariant symmetric bilinear form on $\frg$.
Then $\langle\cdot,\cdot\rangle$ is invariant.
\end{thm}

The assumption that $\met$ is symmetric is crucial for
Theorem \ref{thm_nilinvariant},
since, in general, nil-invariance of a
bilinear form on $\frg$ does not imply its invariance.
Zwart and Boothby \cite[Section 7]{ZB} provide an example
of a skew-symmetric nil-invariant form on a solvable Lie algebra
which is not invariant.\\

An application of Theorem \ref{thm:density} and
Theorem \ref{thm_nilinvariant} is the following:

\begin{cor}\label{cor:AdHinvariant}
Let $G$ be a solvable Lie group, $H$ a uniform subgroup of $G$
and $\g$ a left-invariant pseudo-Riemannian metric on $G$ which is
right-invariant under $H$.
Then $\g$ is bi-invariant.
\end{cor}

Indeed, the left-invariant metric $\g$ induces an inner product
$\met$ on the Lie algebra $\frg$ of $G$.
The right-invariance under $H$ of $\g$ implies that $\met$
is $\Ad(H)$-invariant. As $H$ is uniform in $G$, the density
Theorem \ref{thm:density} implies that $\met$ is nil-invariant.
By Theorem \ref{thm_nilinvariant}, $\met$ is invariant on $\frg$
and thus $\g$ is a bi-invariant metric on $G$.\\

Compact homogeneous spaces for solvable Lie groups are
aspherical manifolds. So as a natural generalization one can
study compact aspherical homogeneous spaces.
In Section \ref{sec:invarvolume} we prove the following theorem:

\begin{thm}\label{thm:invarvolume}
Let $L$ be a connected Lie group that acts almost effectively and transitively on the compact aspherical manifold $M$.  Assume further that $L$ preserves a finite Borel measure on $M$. 
If the fundamental group of $M$ is solvable, then $L$ is solvable.
\end{thm}

\vspace*{1ex}


\subsection*{Notations and conventions}
The identity element of a group $G$ is denoted by $e$.
If $A$ and $B$ are subsets of $G$,  we put  $A \cdot  B = \{ a b \mid a \in A, b \in B \}$.

Let $H$ be a subgroup of $G$.
We write
$\Ad_{\frg}(H)$ for the
adjoint representation of $H$ on the Lie algebra $\frg$ of $G$,
to distinguish it from the adjoint representation $\Ad(H)$ on
its own Lie algebra $\frh$.

A subgroup $\aG$ of $\GL_n(\CC)$ is called a
\emph{linear algebraic group} if it is the solution set of
a system of polynomial equations. We say $\aG$ is
\emph{$K$-defined},
where $K$ is subfield of $\CC$, if the polynomial equations
defining $\aG$ have coefficients in $K$. The
\emph{$K$-points} of $\aG$ are the elements of
$\aG_K=\aG\cap\GL_n(K)$.
A group $\raG = \aG_{\RR}$ is called a \emph{real algebraic group}
if it consists of the $\RR$-points of an
$\RR$-defined linear algebraic group $\aG$.

We let $\raG^{\circ}$ denote the connected component of the
identity of $\raG$ with respect to the Zariski topology,
and $\raG_{\circ}$ the connected component of the identity
with respect to the natural Euclidean topology.
Note that $\raG_{\circ} \subset \raG^{\circ}$.

If $M \subset \aG$ is a subset, $\ac{M}$ denotes the closure of 
$M$ in the Zariski topology.
If $G$ is a Lie group with subgroup $H$, then
we say $H$ is \emph{Zariski-dense} in $G$ if
$\ac{\Ad_{\frg}(H)}=\ac{\Ad_{\frg}(G)}$.



%

\subsection*{Acknowledgements}
Wolfgang Globke was supported by the Australian Research Council grants {\small DP120104582} and {\small DE150101647}.
He would also like to thank the Mathematical Institute of the
University of G\"ottingen,
where part of this work was carried out, for its hospitality
and support.



\section{Review of Jordan decompositions}
\label{sec_jordan}

In this section we recall some facts on the Jordan decomposition
of endomorphisms and the Jordan decomposition in a linear
algebraic group.
Proofs can be found in Borel \cite[Chapter 4]{borel}.

\subsection{The additive Jordan decomposition}

Let $A$ be an endomorphism of a finite-dimensional real
vector space $V$.
There exist a unique semisimple endomorphism $A_{\rm ss}$
(that is, diagonalizable over $\CC$)
and a unique nilpotent endomorphism $A_{\rm n}$ of $V$ such that
\[
[A_{\rm ss},A_{\rm n}]=0
\quad\text{ and }\quad
A = A_{\rm ss} + A_{\rm n}.
\]
This is the \emph{additive Jordan decomposition} of $A$.

Moreover, there exist polynomials $P,Q\in\RR[x]$ with
constant term $0$ such that
\[
P(A)=A_{\rm ss},\quad Q(A)=A_{\rm n}.
\]
$P$ and $Q$ can be chosen as real polynomials.
The fact that the constant term in $P$ and $Q$ is $0$ implies
\[
\im A_{\rm ss}\subset\im A,\quad
\im A_{\rm n}\subset\im A.
\]
In particluar, any $A$-invariant subspace $U$ of $V$ is also
$A_{\rm ss}$- and $A_{\rm n}$-invariant.
The Jordan decomposition of $A$ induces those of $A|_U$ and
$A_{V/U}$.

Since $A_{\rm ss}$ is semisimple,
\[
V = \ker A_{\rm ss} \oplus \im A_{\rm ss}.
\]

\subsection{The multiplicative Jordan decomposition}

Let $g$ be an automorphism of a finite-dimensional real vector
space $V$. Set
\[
g_{\rm u} = \I - g_{\rm ss}^{-1}g_{\rm n}.
\]
Then $g_{\rm u}$ is unipotent (that is, $\I-g_{\rm u}$ is
nilpotent),
\[
[g_{\rm ss},g_{\rm u}]=0
\quad
\text{ and }
\quad
g = g_{\rm ss}\cdot g_{\rm u}.
\]
This is the \emph{multiplicative Jordan decomposition} of $g$.
The elements $g_{\rm ss}$ and $g_{\rm u}$ are uniquely
determined by these properties.
Any $g$-invariant subspace of $V$ is invariant under
$g_{\rm u}$ as well.

\subsection{The Jordan decomposition in an algebraic group}


\begin{thm}\label{thm_jordan}
Let $\aG$ be a linear algebraic group.
For $g\in\aG$, let $g = g_{\rm u}\cdot g_{\rm ss}$
denote its multiplicative Jordan decomposition.
Then $g_{\rm u},g_{\rm ss}\in\aG$, and
if $g\in\aG_{\RR}$, then also $g_{\rm u},g_{\rm ss}\in\aG_{\RR}$.
If $\phi:\aG\to\aH$ is a morphism of linear algebraic groups,
then $\phi(g_{\rm ss})=\phi(g)_{\rm ss}$ and
$\phi(g_{\rm u})=\phi(g)_{\rm u}$ for all $g\in\aG$.
\end{thm}

For a subset $M\subset\aG$ we write
$M_{\rm u}=\{g_{\rm u}\mid g\in M\}$
and $M_{\rm ss}=\{g_{\rm ss}\mid g\in M\}$.
Let $\uni(\aG)=\{g\in\aG\mid g=g_{\rm u}\}$ denote the
collection of the unipotent elements in $\aG$.
The \emph{unipotent radical} $\Uni(\aG)$ of $\aG$ is the maximal
normal subgroup consisting of unipotent elements.
A connected subgroup $\aT\subset\aG$ consisting of semisimple
elements is called a is called a \emph{torus}.

\section{The density theorem for solvable Lie groups}
\label{sec_density}

For a solvable linear algebraic group $\aG$ defined over $\RR$, 
let $\aG_\spl$ denote the maximal $\RR$-split connected subgroup of $\aG$. This means that $\aG_\spl$ is the maximal connected subgroup trigonalizable over the reals. For  a real algebraic group  $\raA = \aG_{\RR}$ its maximal trigonalizable subgroup  is $  \raA_{\spl} =  \raA \cap \aG_\spl$. 
Let $\aT$ be a torus 
defined over $\RR$. Then $\aT$ is called \emph{anisotropic} 
if $\ag{T}_\spl=\one$. Equivalently,  $\aT$ is anisotropic if its group of real points $\raT=\aT_{\RR}$ is compact. 
Every torus defined over 
$\RR$ 
has a decomposition into subgroups
$ \aT=\aT_\spl\cdot \aT_{\rm c}$, 
where $ \aT_{\rm c}$ is a maximal anisotropic torus defined over
$\RR$ and $\aT_\spl \cap \aT_{\rm c}$ is finite.
Moreover, if $\aT \leq \aG$ is a maximal torus defined over $\RR$
and $\aU$ is the unipotent radical of $\aG$,
then there is a direct product decomposition
\[
\aG_{\spl} = \aU \cdot \aT_{\spl}.
\]
Note also that the split part $\aG_{\spl}$ is preserved under
morphisms of algebraic groups which are defined over $\RR$. 
See Borel \cite[\textsection 15]{borel} for more background.\\

The purpose of this section is to prove:

{
\renewcommand{\thethm}{\ref{thm:density}}
\begin{thm}
Let $G$ be a connected solvable Lie group
and $H$ a uniform subgroup,
and let $\rho:G\to\GL(V)$ be a representation on a
finite-dimensional real vector space $V$. 
Let $\ra{A}$ denote the Zariski closure of $\rho(G)$ in
$\GL(V)$. 
Then 
\[
\ac{\rho(H)} \, \supseteq \,  \raA_\spl \; .
\]
\end{thm}
\addtocounter{thm}{-1}
}

Before we give  the main part of the proof, we add an important observation:

\begin{lem}\label{lem:probmeasure}
Let $G$ be a connected solvable Lie group and $H$ a uniform
subgroup.
Then $G/H$ admits a $G$-invariant finite Borel measure.
\end{lem}    
\begin{proof}
Let $\Delta_{H} = |\det\Ad_{\frh}|: H \to \RR $ be
the modular character of $H$, and
$\Delta_{G}|_{H} = |\det\Ad_{\frg}|_{H}: H \to \RR $ the restriction of the modular character of $G$ to $H$.  To show that there exists an invariant measure on $G/H$ it is sufficient (cf.~Raghunathan \cite[1.4 Lemma]{raghunathan}) to show that $\Delta_{H} = \Delta_{G}|_{H}$.

Let $N$ be the nilradical of $G$ and $\frn$ its Lie algebra.
Since $[G,G] \subset N$, $$ \Delta_{G} = |\det\Ad_{\frn}| \, \text{ and } \,  \Delta_{H} = |\det\Ad_{\frh \cap \frn}| \; . $$
Now $H \cap N$ is a uniform subgroup in $N$ by Mostow's theorem \cite[\textsection 5]{mostow0},
$H_{\circ} \cap N$ is a normal subgroup of $N$, and the projection of $H \cap N$  to $N/(H_{\circ} \cap N)$ is a uniform lattice.
We compute
\[
\Delta_{G}|_{H}
=|\det \Ad_{\frn}|_{H}
=|\det \Ad_{\frh\cap \frn}|_{H}\cdot |\det \Ad_{\frn/(\frh \cap \frn)} |_{H}
=\Delta_{H}\cdot 1
= \Delta_{H}.
\]
Note that the second factor is  $\equiv 1$,  since the adjoint of $H$ preserves an integral lattice in ${\frn /(\frh \cap \frn)}$.
Since $G/H$ is compact, any invariant Borel measure is finite. 
\end{proof}

\begin{proof}[Proof of Theorem \ref{thm:density}]
Let $\ag{A}$ be an $\RR$-defined  solvable linear algebraic
group which contains a solvable Lie subgroup 
$G \leq \ra{A} = \ag{A}_{\RR}$ as a Zariski-dense subgroup.
Let $H \leq G$ be a uniform subgroup and $\ag{H}$ the
Zariski closure of $H$.
By a Theorem of Chevalley (see Borel \cite[5.1 Theorem]{borel}),
there exists a complex vector space $W$,
with real structure $U = W_{\RR}$,
a linear representation
$\ag{A} \to \GL(W)$, which is defined over $\RR$, such that
$\ag{H}$ is the stabilizer of a line $[x] \in\PP(W)$, where $x \in U$. We may also assume that the representation is minimal in the following sense: the orbit $G \cdot {x}$ is not contained
in a proper subspace $W_{0}$ of $W$. 

Since $G/H$ has a $G$-invariant probability measure
(Lemma \ref{lem:probmeasure})
and maps into $\PP(U)$ via the orbit map (of the above representation on $U$)  at $[x]$,
there exists a $G$-invariant probability measure on $\PP(U)$.
In view of the minimality property, Furstenberg's Lemma,
see Zimmer \cite[3.2.2 Corollary]{zimmer},
asserts that the stabilizer of this
measure in $\PGL(U)$ is compact.

Therefore,  the (Euclidean closure of the) image
of $G$ is a compact subgroup of real points in
the image $\ag{B}$ of  $\ag{A}$ in $\PGL(W)$, and
it is also Zariski-dense in $\ag{B}$, since $G$ is dense in $\ag{A}$.
It follows that $\ag{B}$ is an anisotropic torus, that is,
$\ag{B}_\spl=\one$.
Note that the homomorphism of algebraic groups 
$\ag{A}\to\ag{B}$ is defined over $\RR$ and maps
$\ag{A}_{\spl}$ to $\ag{B}_{\spl}$. 
Thus its kernel $\ag{K}$ 
contains the maximal $\RR$-split
connected subgroup
$\ag{A}_{\spl}$ of $\ag{A}$.
Since $\ag{K} \leq \ag{H}$ by construction,
Theorem \ref{thm:density} follows.
\end{proof} 



\section{Abelian modules with a skew pairing}
\label{sec_modules}

Let $\fra$ be a real abelian Lie algebra and let $(V,\rho)$ be an
$\fra$-module. The module $(V,\rho)$ is called \emph{nilpotent} if all transformations $\rho(A)$, $A \in \fra$, are nilpotent.
A bilinear map $\langle\cdot,\cdot\rangle:V\times\fra\to\RR$  such that 
\[
\langle \rho(A)v,B\rangle=-\langle\rho(B)v,A\rangle
\text{ for all $A,B\in\fra$, $v\in V$}
\]
will be called a \emph{skew pairing} for $(V,\rho)$.

\begin{prop}  \label{prop:skewpair}
Let $\langle\cdot,\cdot\rangle$ be 
a skew pairing for  $(V,\rho)$. Then $(V,\rho)$ is nilpotent
or there exists a submodule $W \neq \zsp$ of $(V,\rho)$,  
which is contained in the radical $\fra^\perp_V  =  \{v\in V\mid \langle v,\cdot\rangle=0\}$ of $V$. 
\end{prop} 
\begin{proof}
Observe that for any $A \in \fra$, $\langle\rho(A)V, A\rangle = 0$.
Suppose there exists $A \in \fra$ such that 
the submodule $W = \rho(A)^{2}V$ is non-zero.  
Let $w = \rho(A) v \in \rho(A)^{2}V$, where $v\in\rho(A)V$. 
Then, for all $B \in \fra$, $\langle w,B \rangle = \langle\rho(A) v,B\rangle = - \langle \rho(B)v,A\rangle = 0$.
The latter term is zero since $\rho(A)V$ is a submodule for $\fra$. Hence, $W$ is contained in $\fra^\perp_V $.
\end{proof}

\begin{cor} \label{cor:skewpair}
Let $\langle\cdot,\cdot\rangle$ be 
a skew pairing for  $(V,\rho)$. If 
$\fra^\perp_V $ contains no non-trivial submodule of  $(V,\rho)$,
then $(V, \rho)$ is nilpotent.\end{cor}


\section{The radical of a nil-invariant scalar product}
\label{sec_h_0}

\subsection{Metric Lie algebras.} \label{sec:mla}
Let $\frg$ be a  Lie algebra and $\met$ a symmetric bilinear form on $\frg$. The pair $(\frg, \met)$ is called a
\emph{metric Lie algebra}, $\met$ is called a \emph{scalar product} (sometimes also metric) on $\frg$.   

The form $\met$ is called \emph{non-degenerate} if
\[
\frr  =  \frg^\perp  = \{ X \in \frg \mid \langle X, \frg\rangle  = \zsp \}
\]
is trivial. The subspace $\frr \subset \frg$ is called the
\emph{metric radical} of $(\frg, \met)$.

The maximal nilpotent ideal $\frn$ of $\frg$ is called
the \emph{nilradical}.

For $X, Y \in \frg$, we write $X \perp Y$ if  $  \langle X, Y \rangle = 0$. Moreover, if $\frv \subset \frg$ is a subspace then $\frv^{\perp} = \{ X \in \frg \mid \langle X, \frv \rangle = \zsp \}$. The subspace $\frv$ is called \emph{totally isotropic} if $\frv \subset \frv^{\perp}$. 
The \emph{signature} of $\met$ is the dimension of a 
maximal totally isotropic subspace.

Assume that $\met$ is non-degenerate. 
Then, given a totally isotropic subspace $\fru$ of $\frg$,
we can find a non-degenerate subspace $\frw$ such that
$\fru^\perp=\frw\oplus\fru$, and a totally isotropic subspace $\frv \subset \frw^\perp$ 
such that $\frv$ is dually paired with $\fru$ by $\met$, see 
\cite[Chapter XV, Lemma 10.1]{lang}. 
The resulting decomposition
\[
\frg = \frv \oplus\frw\oplus\fru 
\]
is called a \emph{Witt decomposition} for $\fru$. 

Let $\varphi: \frg \to \frg$ be a linear map. Then $\met$ is called \emph{$\varphi$-invariant}  if $\varphi$ is skew-symmetric with respect to $\met$, that is, if 
$\langle \varphi X,Y\rangle = -\langle X, \varphi Y\rangle$
for all $X,Y \in\frg$.

We put 
\[
\inv\left(\frg,\met\right)
=  \left\{X\in\frg \mid \langle[X,Y],Z\rangle=-\langle Y,[X,Z]\rangle
\text{ for all } Y,Z\in\frg \right\}.
\]
If $\frh$ is a subspace of $\inv(\frg,\met)$ then we say $\met$ is \emph{$\frh$-invariant}. Moreover, $\met$ is called \emph{invariant}   if 
$\inv(\frg,\met) = \frg$.

\begin{definition}\label{def_nilinvariant}
The metric Lie algebra $(\frg, {\met})$ is called \emph{nil-invariant}
if  $\met$ is invariant under the nilpotent part $\ad(X)_{\rm n}$ in the additive Jordan decomposition of  $\ad(X)$ for all $X\in\frg$.
\end{definition}
\subsection{Nil-invariant metric Lie algebras}

The metric Lie algebra $(\frg, {\met})$ is called \emph{reduced} if the metric radical  $  \frr = \frg^{\perp}$ 
does not contain any non-trivial ideal of $\frg$.
The main result of this section is:

\begin{prop}\label{prop_radicalzero}
Let $\frg$ be a solvable Lie algebra and  $\met$ a 
nil-invariant symmetric bilinear form. If $(\frg, \met)$ is reduced, then the metric radical $\frr$ is zero, that is, the metric $\met$ is non-degenerate. 
\end{prop} 

This implies: 

\begin{cor}\label{cor_radicalzero}
Let $\frg$ be a solvable Lie algebra and  $\met$ a 
nil-invariant symmetric bilinear form. Then the metric radical $\frr$ for $\met$ is an ideal in $\frg$. 
\end{cor}

Furthermore we show: 

\begin{cor}\label{cor_znzg}
Let $\frg$ be a solvable Lie algebra with a nil-invariant
symmetric bilinear form $\met$ and let $\zen(\frg)$ be 
the center of $\frg$. If $(\frg, \met)$ is reduced, then:
\begin{enumerate}
\item
$\zen(\frn)=\zen(\frg)$.  
\item
If $\frg$ is not abelian, then $\zen(\frg)$ contains a non-trivial totally isotropic charac\-teristic ideal of $\frg$. In particular, $\zen(\frg) \neq \zsp$.
\end{enumerate}
\end{cor}

The proofs of Proposition \ref{prop_radicalzero} and Corollary \ref{cor_znzg} will be given in Section \ref{sec_j0}.


%

\subsection{Totally isotropic ideals in $\boldsymbol{\zen(\frn)}$}
\begin{lem} Let $\frr = \frg^\perp$ be the metric radical. Then 
\label{lem_invrad}
\begin{enumerate}
\item 
$[\inv(\frg,\met), \frr]\subseteq \frr$. 
\item 
Let $\frj\subset\inv(\frg,\met)$
be an ideal in $\frg$.
Then $[\frj^\perp,\frj]\subset\frj\cap \frr$.
\end{enumerate}
\end{lem}
\begin{proof}
For the proof of (2)  let  $Y\in\frj^\perp$ and $Z\in\frj$.
Since $\frj$ is an ideal, for any $X\in\frg$,
\[
\langle [Y,Z],X\rangle=-\langle Y,[X,Z]\rangle=0
\]
holds.
So $[Y,Z]\perp \frg$. Hence $[Y,Z]\in\frj\cap\frr$.
\end{proof}

\begin{lem}\label{lem_j} 
Let $\frn$ be an ideal in $\frg$ with $[\frg, \frg] \subset \frn$.
If $\met$ is $\frn$-invariant, then:
\begin{enumerate}
\item
$[\frg,\frn]\perp\zen(\frn)$.
\item
$ \frz(\frn)\cap[\frg,\frn]$
is a totally isotropic ideal in $\frg$.
\end{enumerate}
Let $\frj\subset\zen(\frn)$ be an ideal in $\frg$.
If 
$\met$ is
nil-invariant then:
\begin{enumerate}\setcounter{enumi}{2}
\item
$\frj^\perp$ is an ideal in $\frg$.
\end{enumerate}
\end{lem}
\begin{proof}
Let $Z\in\zen(\frn)$, $X\in\frg$, $Y\in\frn$.
Then $\langle Z,[X,Y]\rangle = -\langle [Z,Y],X\rangle=0$,
which proves (1). Hence, (2) follows. 

For $X\in\frg$, we have $\ad(X)\frj\subset\frj$, as $\frj$
is an ideal.
Then $\ad(X)_{\rm n}\frj\subset\frj$ (see Section \ref{sec_jordan}), and also
\[
\ad(X)_{\rm n}\frj^\perp\subset\frj^\perp,
\]
as $\met$ is invariant under
$\ad(\frg)_{\rm n}$.
For the semisimple part, observe that
\[
\ad(X)_{\rm ss} \frg \subset \ad(X)_{\rm ss} [\frg, \frg] \subset [\frg, \frn]  \subset \frj^\perp.
\]
In particular, this means
$\ad(X)\frj^\perp\subset\frj^\perp$ and thus (3) holds.
\end{proof}

Let $\frj\subset\zen(\frn)$ be a totally isotropic ideal of $(\frg, \met)$. Since $\frj$ is totally isotropic, there exists 
a totally isotropic subspace
$\fra$ of $\frg$ such that 
\begin{equation}
\frg = \fra\oplus\frj^\perp.
\label{eq_decomposition}
\end{equation}
Note that $\met$ induces a dual pairing between 
$\fra$ and   $\frj/(\frj\cap\frr)$. 
%

\begin{lem}\label{lem_ada_abelian}
The restricted adjoint representation  $\ad_{\frg}(\fra)|_{\frj}$ of 
$\fra$ on $\frj$ is abelian.
\end{lem}
\begin{proof}
For all $A,B\in\fra$,
\[
[\ad_{\frg}(A)|_{\frj},\ad_{\frg}(B)|_{\frj}]
=\ad_{\frg}([A,B])|_{\frj}=0,
\]
because $[A,B]\in\frn$ and $\frj\subset\zen(\frn)$.
\end{proof}

\begin{prop}\label{prop_jh}
Let $(\frg, \met)$ be a reduced solvable metric Lie algebra
with metric radical $\frr$. 
If\/ $\met$ is nil-invariant then 
the following hold:
\begin{enumerate}
\item $[ \frj^{\perp}, \frj] = \zsp$.
\item $\frj\cap\frr= \zsp$.
\item $\frg$ acts on $\frj$ by nilpotent transformations.
\end{enumerate}
\end{prop}
\begin{proof}
Since both $ \frj^{\perp}$, $\frj$ are ideals in $\frg$, so
is  $[ \frj^{\perp}, \frj  ]$. 
By (2) of Lemma \ref{lem_invrad}, the ideal
$[ \frj^{\perp}, \frj ]$ is contained in  $\frr$.
Therefore, the reducedness of $\met$ implies (1). 

Consider the pairing $\met: \fra \times \frj \to \RR$. Since $\frj \subset \frn$, nil-invariance implies
\[
\langle [A,X], B \rangle = -  \langle [B,X], A \rangle,
\quad\text{for all $X\in\frj$ and $A,B\in\fra$}.
\]
It follows that $\met$ is a skew pairing with respect to the adjoint representation of $\fra$ on $\frj$ in the sense of Section \ref{sec_modules}. Note further that 
$\frr \cap \frj$ is the radical of this 
skew pairing.
Assume that $U \subset \frr \cap \frj$ satisfies $[ \fra, U ] \subset U$. By (1) above, $U$ is an ideal of $\frg$. Since $U \subset  \frr$, reducedness implies that $U = \zsp$. Thus the assumption of 
Corollary \ref{cor:skewpair} is satisfied.  Corollary 
\ref{cor:skewpair} therefore asserts that $\fra$ acts nilpotently on 
$\frj$. 
Nil-invariance of $\met$ further implies that $[ \fra, U ] \subset U$ for $U = \frj \cap \frr$. Thus $ \frj \cap \frr = \zsp$, and (2) holds.
Corollary \ref{cor:skewpair} together with (1) implies (3).
\end{proof}


\subsection{The characteristic ideal  $\boldsymbol{\zen(\frn)\cap[\frg,\frn]}$.} \label{sec_j0}
Recall that $\frn$ denotes the nilradical of $\frg$.
One key element in our analysis will be the following characteristic ideal of $\frg$:
\begin{equation}
\frj_0 = \zen(\frn) \cap [\frg,\frn].
\label{eq_j}
\end{equation}
A fundamental property is:
\begin{prop} \label{prop_j0}
$\frg$ is abelian if and only if\/ $\frj_0 = \zen(\frn) \cap [\frg,\frn] ={\bf 0}$.
\end{prop}
\begin{proof}
Assume $\frg$ is not abelian. If $\frn$ is not abelian, then
$\frj_0\supset \zen(\frn)\cap[\frn,\frn]\neq{\bf 0}$.
If $\frn$ is abelian, then $\frj_0=[\frg,\frn]$.
Assuming $[\frg,\frn]={\bf 0}$, we find $[\frg,[\frg,\frg]]
={\bf 0}$. So $\frg$ is nilpotent, hence $\frg=\frn$ is abelian, contradicting our assumption. This shows $\frj_0\neq\zsp$. 
\end{proof}

We turn now to the properties of $\frj_0$ with respect to nil-invariant metrics: 

\begin{lem}\label{lem_nh0} Assume that the solvable metric Lie algebra $(\frg, \met)$ has nil-invariant metric $\met$, and
let $\frr$ denote the metric radical of $\met$.
Then:   
\begin{enumerate}
\item
$\frj_0=\zen(\frn)\cap[\frg,\frn]$
is a totally isotropic ideal.
\end{enumerate}
Moreover,  if $(\frg, \met)$ is reduced, then the following hold:
\begin{enumerate}
\setcounter{enumi}{1}
\item
$[\frn,\frr]={\bf 0}$.
\item
$\frr\subset\frz(\frn)$. In particular, $\frr$ is abelian.
\item 
$[\frg,\frr]\subset\frj_0$.
In particular, $\frj_0\oplus\frr$ is an ideal in $\frg$.
\item $[\frj_0^\perp,\frj_0\oplus\frr]={\bf 0}$.
\item $[\frg,\zen(\frn)] =\zsp$.
\end{enumerate}
\end{lem}
\begin{proof}

Nil-invariance implies that (1) holds.
Moreover,  $[\frn, \frr] \subset \frr$, and
hence $\frn$ acts on $\frr$ and $[\frn,\frr]$. 
Since the action of $\frn$ is nilpotent, assuming $[\frn,\frr]\neq{\bf 0}$, there exists a non-zero $Z\in[\frn,\frr]$ such that $\ad(X)Z=0$
for all $X\in\frn$.
Hence $Z\in\frj_0\cap\frr$. But
$\frj_0\cap\frr={\bf 0}$ by Proposition \ref{prop_jh}, a contradiction.
It follows that $[\frn,\frr]={\bf 0}$. Hence (2) holds.

For all $Y\in\frr$ it follows from (2) that
$[Y,\frg]\subset\frn$ implies $[Y,[Y,\frg]]={\bf 0}$.
Hence
$\frr\subset\{Y\in\frg\mid\ad(Y)\text{ is nilpotent}\}=\frn$.
Again by (2), $\frr\subset\zen(\frn)$. Hence (3) holds.
Now (4) is immediate from (3). 

Let $Z\in\frj_0\oplus\frr\subset\zen(\frn)$.
For all $X\in\frg$, $[X,Z]\in\zen(\frn)\cap[\frg,\frn]=\frj_0$.
Now let $Y\in\frj_0^\perp$. Then
\[
\langle [Y,Z],X \rangle=-\langle Y,[X,Z] \rangle = 0,
\]
which means $[Y,Z]\in\frr$.
But then $[Y,Z]\in\frj_0\cap\frr={\bf 0}$. Hence, (5) holds. 

Finally, since $[\frg, \zen(\frn)] \subset \frj_{0}$, (3) of Proposition \ref{prop_jh} implies that $\frg$ acts nilpotently on $\zen(\frn)$.
It then follows that for all $X,Y\in\frg$, $Z\in\zen(\frn)$,
\[
\langle [X,Z],Y\rangle
=\langle\ad(X)_{\rm n} Z,Y\rangle
=-\langle Z,\ad(X)_{\rm n}Y\rangle
=0 .
\]
The latter term is $0$ since $\ad(X)_{\rm n}Y\in[\frg,\frn]$ and
$[\frg,\frn]\perp\zen(\frn)$ by Lemma \ref{lem_j}. Hence,
$[\frg,\zen(\frn)] \subset \frr$ and since $(\frg,\met)$ is
reduced, (6) holds.
\qedhere
\end{proof}

%
\begin{proof}[Proof of Proposition \ref{prop_radicalzero}] 
 We decompose $\frg = \fra \oplus \frj_0^{\perp}$ as in
\eqref{eq_decomposition}. 
By Proposition \ref{prop_jh}, $\ad(\fra)$ acts on $\frj_0$ by
nilpotent operators. By (4) of Lemma \ref{lem_nh0}, $[\fra,\frr]\subset\frj_0$.
So $\ad(\fra)$ acts on $\frj_0\oplus\frr$ by nilpotent operators.

For all $A,B\in\fra$ and $H\in\frr$,
we thus find
\[
\langle \ad(A)H,B\rangle
= \langle \ad(A)_{\rm n}H,B\rangle
= -\langle H,\ad(A)_{\rm n}B\rangle
=0.
\]
Hence $\ad(\fra)\frr\subset\fra^\perp\cap\frj_0=\frr\cap\frj_0={\bf 0}$.
By (5) of Lemma \ref{lem_nh0}, $[\frj_0^\perp,\frr]={\bf 0}$.
Therefore, $[\frg,\frr]={\bf 0}$.
So $\frr$ is an ideal in $\frg$ and
thus $\frr={\bf 0}$ by reducedness.
\end{proof}


\begin{proof}[Proof of Corollary \ref{cor_znzg}] 
Assertion (1) is implied by (6) of Lemma \ref{lem_nh0}.
If $\frg$ is not abelian, then $\frj_0$ is non-trivial by Proposition \ref{prop_j0}. It is contained in $\zen(\frg)$ by (1).  Hence,  (2) follows.
\end{proof}

\section{Reduction by a totally isotropic central ideal}
\label{sec_reduction}

Let $(\frg, \met)$ be a metric Lie algebra, where $\frg$ is solvable and $\met$ is a nil-invariant \emph{non-degenerate} symmetric bilinear form. We show that $\met$ is invariant. 

\subsection{Reduction}
Let $\frj \subset\zen(\frg)$ be a totally 
isotropic ideal in  $\frg$ which is central. Then $\frj^\perp$ is an ideal in $\frg$. In particular, we can consider the quotient Lie algebra
 $$ \bar \frg = \frj^\perp / \, \frj  \; .  $$
Since $\frj$ is totally isotropic, $\bar \frg$ inherits a non-degenerate symmetric bilinear form from $\frj^{\perp}$. The metric Lie algebra $(\bar \frg, \langle\cdot,\cdot\rangle)$ 
will be called the \emph{reduction}  of $(\frg, \langle\cdot,\cdot\rangle)$ by $\frj$.  

We may choose a totally isotropic vector subspace $\fra$ of $\frg$  to obtain a Witt-decomposition 
\begin{equation} \label{eq_decompositionW}
\frg = \fra \oplus\frw\oplus\frj  \; , 
\end{equation}
where 
$\frw$ is a non-degenerate subspace orthogonal to $\fra$ and $\frj$. 

For all $X \in \frg$, we write $X = X_{\fra} + X_{\frw} + X_{\frj}$ with respect to \eqref{eq_decompositionW}.  In what follows we shall frequently indentify ${\frw}$ with the underlying vector space of $\bar \frg$. Thus for $X \in \frj^\perp$, the projection $\bar X$ of $X$ to $\frg$ may also be considered as the element $X_{\frw} \in\frw$. Similarly, 
$[ \bar X, \bar Y ]_{\bar \frg} = [X,Y]_{\frw}$ for $X,Y \in  \frj^\perp$ is the Lie bracket in $\bar \frg$. 
The Lie product in $\frg$
thus gives rise to the following equations:

For all $X,Y \in \frj^\perp$, 
\begin{equation}
\label{eq_defomega}
[X,Y] =  [\bar X,\bar Y]_{\bar \frg} + \omega(\bar{X}, \bar{Y}),
\end{equation} 
where $\omega\in\Z^2(\bar \frg, \frj)$ is a $2$-cocycle.

For all $A \in \fra$, $X \in \frj^{\perp}$,
\begin{equation}
[A , X]  = \bar{A}\, \bar{X} + \xi_{A}(\bar{X}),
\label{eq_A_action}
\end{equation}
where $\xi_{A}: \bar \frg \to \frj$ is a linear map, 
and $\bar{A}$ is the derivation of $\bar{\frg}$ 
induced by $\ad(A)$.

\begin{remark}\label{rem_der}
Recall that any derivation of $\frg$ maps $\frg$ to the nilradical $\frn$ (Jacobson \cite[Theorem III.7]{jacobson}).
If $S$ is a semi\-simple derivation, this implies
\begin{equation*}
S\frg = S\frn \subseteq\frn.
\label{eq_der_ss}
\end{equation*}
In particular, this holds for derivations of the form $S=\ad(X)_{\rm ss}$, $X\in\frg$.
\end{remark}

In a split situation, the maps $\xi_{A}$ vanish: 
\begin{lem}\label{lem_An1}
Assume that $[\fra, \fra] = \zsp$ (that is, $\fra$ is an abelian 
subalgebra). Then $[\fra, \frg]$ is contained in $\fra^{\perp}$. 
In particular, $\xi_{A} = 0$ for all $A \in \fra$. 
\end{lem}

\begin{proof}
Let $A \in \fra$. Note that $ \ad(A)_{\rm ss}\frg =  \ad(A)_{\rm ss} \frn$ is contained in $\ad(A)^{2}\frn$, where $\frn$ is the nilradical.  
The $\frn$-invariance implies that the pairing $\met: \fra \times \frn \to \RR$ is skew with respect to the representation  $A \mapsto \ad(A) {|_\frn}$. Thus the proof of Proposition \ref{prop:skewpair} shows that $\ad(A)^{2}\frn \subset \fra ^{\perp}$, and hence
$ \ad(A)_{\rm ss}\frg \subset \fra ^{\perp}$. 
Now let $X \in \frg$, $B \in \fra$. Then using $\ad(A)_{\rm n }$ is skew and  $\ad(A)_{\rm n } B = 0$, we obtain $\langle [ A, X], B \rangle \, = \langle  \ad(A)_{\rm n }X , B \rangle    \,  =  -  \,  \langle  \ad(A)_{\rm n } B ,  X \rangle  \, =  0 $. 
\end{proof}


If the reduction $(\bar \frg , \met)$ has invariant metric, the derivation $\bar A$ and the extension cocycle $\omega$ determine each other: 
\begin{prop}  \label{prop_A_and_O}
Let $\frj \subset \zen(\frg)$ be a totally isotropic ideal. Assume that the reduction $(\bar \frg , \met)$ with respect to $\frj$ has an invariant metric.  Then, for all $X,Y\in \frj^{\perp}$, $A\in\fra$, we have 
\begin{equation}
\label{eq_A_and_O}
  \langle\bar{A}\,  \bar X,  Y\rangle
= \langle \omega(\bar X, \bar Y), A \rangle   \; .
\end{equation}
\end{prop}
\begin{proof} Let  $\ad(X)=\ad(X)_{\rm ss}+\ad(X)_{\rm n}$
be the Jordan decomposition. Observe that  
$\frg$ decomposes as
$
\frg = \im\ad(X)_{\rm ss}\oplus\ker\ad(X)_{\rm ss} $.
First, assume $Y\in\ker\ad(X)_{\rm ss}$.
We write $A$ as $A=A_0+A_1$ with $A_0\in\ker\ad(X)_{\rm ss}$
and $A_1\in\im\ad(X)_{\rm ss}$. Then
\begin{align*}
\langle [A,X],Y\rangle &= \langle [A_0,X],Y\rangle +\langle [A_1,X],Y\rangle \\
&=- \, \langle \ad(X)_{\rm n}A_0,Y\rangle + \langle A_1,[X,Y]\rangle \\
&=\langle A_0,\ad(X)_{\rm n}Y\rangle + \langle A_1,[X,Y]\rangle \\
&=\langle A_0,[X,Y]\rangle + \langle A_1,[X,Y]\rangle \\
&=\langle A,[X,Y]\rangle.
\end{align*}
For the second equality,  we used that $A_{1}\in[X,\frg]\subset\frj^{\perp}$. Then the assumption that the metric $\langle\cdot,\cdot\rangle$ on ${\bar{\frg}}$ is invariant can be applied. 

Next assume $Y\in\im\ad(X)_{\rm ss}$. 
Then there exists $W\in \frn$
such that $Y=[X,W]$, in particular $Y \in \frn$
(Remark \ref{rem_der}).
Then
\begin{align*}
\langle [A,X],Y\rangle &= \langle [A,X],[X,W]\rangle \\
&=  - \, \langle [[A,X],W],X\rangle \\
&= \langle [[W,A],X],X\rangle + \langle [Y,A],X\rangle \\
&= 0 - \langle A, [Y,X]\rangle \\
&= \langle A, [X,Y]\rangle.
\end{align*}
We used the fact that $[W,A]\in\frn$ to find
$\langle [[W,A],X],X\rangle=0$.
\end{proof} 

\subsection{Invariance of the metric}
Every non-abelian 
metric Lie algebra $(\frg, \met)$ with nil-invariant symmetric bilinear form $\met$ admits a non-trivial totally isotropic 
and central ideal $\frj$, see Corollary \ref{cor_znzg}. Therefore,
$(\frg, \met)$ reduces 
to a metric Lie algebra  $(\bar \frg, \met)$ of lower dimension.
Iterating this procedure we obtain: 

\begin{prop} \label{prop_completereduction}
Let $(\frg, \met)$ be a solvable metric 
Lie algebra with  nil-invariant non-degenerate symmetric bilinear form $\met$. After a finite sequence of successive reductions 
with respect to one-dim\-en\-sion\-al totally isotropic and central ideals,  
$(\frg, \met)$ reduces to an abelian metric Lie algebra 
with positive definite metric $\met$.
\end{prop}
\begin{proof}
We can apply the reduction again to $(\bar \frg, \met)$ to obtain a sequence of 
successive reductions. For this, note that the nil-invariance
property is inherited in each reduction step. The process 
terminates if and only if the reduction is abelian with a positive definite metric, for otherwise it can be 
further reduced.  
\end{proof}

If a reduction $(\bar{\frg},\met)$ has positive definite
scalar product, then it cannot be reduced further.
In this case we call it a \emph{complete reduction}.
From Proposition \ref{prop_completereduction} we immediately
obtain:

\begin{cor}\label{cor_definite_base}
If $\dim\frg=n$ and the signature of $\met$ is $s$, then
the unique complete reduction of $(\frg,\met)$ is isometric to
$(\RR^{n-2s},\met_+)$,
where $\met_+$ denotes the canonical positive definite
scalar product on $\RR^{n-2s}$.
\end{cor}

We further deduce: 

\begin{cor}\label{cor_invariant}
Let $(\frg, \met)$ be a solvable metric 
Lie algebra with  nil-invariant non-degenerate symmetric bilinear form  $\met$. Then $\met$ is invariant.
\end{cor}
\begin{proof} 
After $\ell$ sucessive reduction steps, the reduction 
$(\frg_{\ell}, \met) = (\fra, \met)$ is abelian with 
positive definite symmetric bilinear form. Then $\met$ is clearly invariant on $\fra$, since $\fra$ is abelian.  We assume now inductively that  the symmetric bilinear form $\met$ on $ \frg_{k+1}$ is invariant. Thus both Lemma \ref{lem_An1} and equation \eqref{eq_A_and_O} apply to the $k$-th reduction step.  It is then easily verified  using equations \eqref{eq_defomega} and \eqref{eq_A_action} (as in the proof of Proposition \ref{prop_A_and_O}) that the metric $\met$ on the Lie algebra $ \frg_{k}$
is invariant.\footnote{Indeed, it follows that $( \frg_{k}, \met)$ is obtained from  $( \frg_{k+1}, \met)$ by the double extension procedure as defined by Medina and Revoy \cite{MR}.} 
\end{proof} 

\subsection{The main theorem on nil-invariant scalar products}

{
\renewcommand{\thethm}{\ref{thm_nilinvariant}}
\begin{thm}
Let $\frg$ be a solvable Lie algebra and  $\met$ a 
nil-invariant symmetric bilinear form on $\frg$.
Then $\langle\cdot,\cdot\rangle$ is invariant.
\end{thm}
\addtocounter{thm}{-1}
}

\begin{proof}
Let $\frr$ be the metric radical of the nil-invariant
form $\met$ on the solvable Lie algebra $\frg$.
By Corollary \ref{cor_radicalzero}, $\frr$ is an ideal in $\frg$.
So $\met$ induces a non-degenerate symmetric bilinear form,
also denoted by $\met$, on $\frg/\frr$.
The invariance of $\met$ on $\frg/\frr$ is given by
Corollary \ref{cor_invariant}.
It is then straightforward to check that the original bilinear
form $\met$ on $\frg$ is invariant as well.
\end{proof}

\section{Proofs of Theorems \ref{mthm_lattice} and \ref{mthm_invariant}}
\label{sec_proofs}

Let $M$ be a compact pseudo-Riemannian manifold and
$G$ a solvable connected Lie group of isometries which acts
transitively on $M$. Let $x \in M$
and $H=G_x$ denote the stabilizer of $x$.
Then $H$ is a uniform subgroup of $G$.

Let $\frg$ and $\frh$ denote the Lie algebras of $G$ and $H$,
respectively.
The pull-back of the pseudo-Riemannian metric $\g$ on $M$ via the orbit map at $x$ is 
a left-invariant symmetric bilinear tensor on $G$ and
restricts to a symmetric
bilinear form $\met$ on $\frg$.
Since $\g$ is non-degenerate, the metric radical $\frr$ of
$\met$ in $\frg$ (as defined in Section \ref{sec:mla}) is precisely the Lie algebra $\frh$ of $H$.
As $G$ is a group of isometries, $G$ acts effectively on $M$.
In particular, $H$ does not contain any connected subgroup which is normal in $G$. Therefore,  the metric radical  $\frr = \frh$ 
does not contain any non-trivial ideal of $\frg$. That is, the metric Lie algebra $(\frg, \met)$ is reduced in the sense of Section \ref{sec_h_0}.

Note  that, since $H$ is the isotropy group at $x$, $\Ad_\frg(H)$ acts by linear isometries of $\met$.
Let $\raA$ denote the Zariski closure of $\Ad_\frg(G)$ in
$\GL(\frg)$.
The density Theorem \ref{thm:density} implies, in particular, 
that the Zariski closure of ${\Ad_\frg(H)}$  contains all unipotent elements of $\raA$. Since $\met$ is preserved by $H$,
its Zariski closure also acts by isometries. Taking derivatives it follows that, for all $X \in \frg$, the nilpotent parts $\ad(X)_{\rm n}$ (in the Jordan decomposition
of $\ad(X)$) are skew-symmetric with respect to $\met$.
This means $\met$ is nil-invariant in the
sense of Definition \ref{def_nilinvariant}.

\begin{proof}[Proof of Theorem \ref{mthm_lattice}]
Since $\met$ is nil-invariant and reduced, Proposition \ref{prop_radicalzero} implies that $\frh =\frr={\bf0}$.  
Hence $H$ is a discrete subgroup of $G$, which implies that $G$ acts almost freely on $M$.
\end{proof}


\begin{proof}[Proof of Theorem \ref{mthm_invariant}]
Since $H$ is discrete by Theorem \ref{mthm_lattice}, the pull-back $\g_G$ of the pseudo-Riemannian metric $\g$ on $M$ is a pseudo-Riemannian metric on $G$. Since $\met$ is nil-invariant, Theorem  \ref{thm_nilinvariant} implies that $\met$ is invariant by all of $\frg$. That is, all operators $\ad(X)$, $X \in \frg$, are skew-symmetric with respect to $\met$.
This implies that the pull-back metric  $\g_G$ is
bi-invariant (cf.~O'Neill \cite[Proposition 11.9]{oneill}).
\end{proof}


\section{Finite invariant measure and solvable fundamental group}
\label{sec:invarvolume}


In this section, we will prove:
{
\renewcommand{\thethm}{\ref{thm:invarvolume}}
\begin{thm}
Let $L$ be a connected Lie group that acts almost effectively and transitively on the compact aspherical manifold $M$.  Assume further that $L$ preserves a finite Borel measure on $M$. 
If the fundamental group of $M$ is solvable, then $L$ is solvable.
\end{thm}
\addtocounter{thm}{-1}
}

Clearly, if $L$ preserves a pseudo-Riemannian metric on $M$, there exists an invariant Borel measure. Therefore, Theorem \ref{thm:invarvolume} implies the first assertion of Corollary \ref{mcor_isoM} in the introduction, namely that the identity component of the isometry group of a homogeneous pseudo-Riemannian metric on $M$ is solvable.

\subsection{Aspherical homogeneous spaces with invariant volume}

Consider a compact aspherical homogeneous space
$M= L/H$ where $L$ is a simply connected Lie group which acts almost effectively on $M$.  Therefore, we can write $L$  as a semidirect product 
\[
L = R  \rtimes S \; ,
\]
where $R$ is the solvable radical of $L$ and $S$ is a Levi subgroup. Recall that a \emph{Levi subgroup}  of $L$  is a maximal connected semisimple subgroup. A basic observation on such spaces is the following:

\begin{lem}\label{lem:SL2}
The Levi subgroup $S$ is  isomorphic to $\tilde \SL_{2}(\RR)^{\ell}$.
\end{lem}
\begin{proof}
The only compact connected groups that act almost effectively on compact aspherical manifolds are tori
(cf.~Conner and Raymond \cite{CR}).
As a consequence, the maximal compact subgroup in the semisimple
group $S$ is a torus. It follows that the universal covering group
$\tilde{S}$ of $S$ is isomorphic to $ \tilde{\SL}_{2}(\RR)^{\ell}$.
Since $S$ as above is simply connected, $S$ is isomorphic to $\tilde \SL_{2}(\RR)^{\ell}$.
\end{proof}

Let $\p: L \to S$ denote the projection homomorphism. We shall prove: 

\begin{thm} \label{thm:asph_invarvol}
Assume that $L$ preserves a finite Borel measure on $M$.
Then $H \cap R$ is uniform in $R$ and the projection $\p(H)$ is a discrete uniform subgroup in $S$. 
\end{thm}

Observe that Theorem \ref{thm:asph_invarvol} implies 
Theorem \ref{thm:invarvolume}.
Indeed, assume that  $\pi_{1}(M) = H / H_{\circ}$ is solvable. Since $\p(H)$ is a discrete subgroup of $S$, $H_{\circ}$ is contained in $R$, and $\p(H)$ is solvable and a uniform lattice in $S$. This implies $S = \one$. Therefore, $L = R$ is solvable. This proves Theorem \ref{thm:invarvolume}.\\ 

The remainder of this chapter is devoted to proving Theorem \ref{thm:asph_invarvol}.

\subsection[Parabolic subgroups  of $\SL_{2}(\RR)$ and related groups]{Parabolic subgroups and uniform subgroups of\/ $\boldsymbol{\tilde \SL_{2}(\RR)}$}

We consider the subgroups  $A, N\subset\SL_{2}(\RR)$  of diagonal matrices
with positive entries and  of unipotent upper-triangular
matrices, respectively.
Let
\[
\tilde \SL_{2}(\RR)\to\SL_2(\RR)
\]
be the universal covering group of $\SL_{2}(\RR)$.
Note that the kernel  of this  map
is an index two subgroup of
the center $\mathsf{Z}$ of $\tilde \SL_{2}(\RR)$,  
and $\mathsf{Z}$ is a subgroup of $\tilde{K}$, where 
$\tilde{K}$ is the preimage of the subgroup $K=\SO_2$. 
Every connected proper subgroup of $\tilde{\SL}_2(\RR)$ is
conjugate to one of $\tilde{K},A,N$ or $AN$, and there is an \emph{Iwasawa decomposition} of the form 
\[
\tilde \SL_{2}(\RR) = \tilde{K} \cdot A  N.
\]  
Our arguments will be based on:

\begin{lem}\label{lem:SL2_dec}
Let $H$ be a uniform subgroup of\/ $\tilde{\SL}_2(\RR)$ such
that $H$ contains a non-trivial connected solvable normal subgroup.
Then:
\begin{enumerate}
\item
The identity component $H_{\circ}$ of $H$ is conjugate to $N$ or $AN$.
\item
The quotient space $\tilde{\SL}_2(\RR)/H$ has no Borel measure which
is invariant by\/ $\tilde{\SL}_2(\RR)$.  
\end{enumerate}
\end{lem}
\begin{proof} 
Evidently, $N$ or $A N$ are the only subgroups of $\tilde{\SL}_2(\RR)$ whose normalizer is uniform. Indeed,  then $H$ is contained in $\mathsf{Z} \cdot AN$.  This proves (1).

Using (1), we compute the modular character 
$\Delta_H: H  \to \RR^{>0}$ of $H$ as
\[
\Delta_H =|\det \Ad_{\frh}| 
=|\det\Ad_{\frn}| \; .
\]
The kernel of $\Delta_H$ is therefore contained in 
$\mathsf{Z} \cdot N$. Since $H$ is uniform in $\mathsf{Z} \cdot AN$,
there exists $h \in H$ with $\Delta_H(h) \neq 1$. Recall that 
${\SL}_2(\RR)$ is a unimodular Lie group.
This shows that
$\Delta_H \not \equiv \Delta_{\tilde{\SL}_2(\RR)}|_{H} \equiv 1$.
Therefore, $\tilde{\SL}_2(\RR)/H$ has no  finite invariant Borel
measure.
\end{proof}


If $S$ is locally isomorphic to  $\SL_2(\RR)^{\ell}$ then a
connected  subgroup is called \emph{minimal parabolic} if it is
locally isomorphic to a conjugate of the subgroup $(AN)^{\ell}$.
Moreover, a connected subgroup 
$P \leq S$  is called \emph{parabolic} if $P$ contains a minimal 
parabolic subgroup. 

\subsection{Proof of Theorem \ref{thm:asph_invarvol}} 

\begin{lem}\label{lem:SmodC}
Let $C \leq S$ be a uniform subgroup such that the identity component $C_{\circ}$ is solvable. Then: 
\begin{enumerate}
\item $C_{\circ}$ is contained in a minimal parabolic subgroup of $S$.
\item If $S/C$ has a finite Borel measure which is invariant by $S$ then $C$ is discrete. 
\end{enumerate}
\end{lem}
\begin{proof} We may consider the projection of $C$ to the factors of $S$. Applying (1) of Lemma \ref{lem:SL2_dec} then implies that $C_{\circ}$ is contained in a minimal parabolic subgroup of $S$.  This shows (1). 

Consider any projection of $C$ to one of the simple factors $\tilde{\SL}_2(\RR)$ of $S$. The image of $C$ is contained in a uniform 
subgroup $H$ in $\tilde{\SL}_2(\RR)$, and we obtain an 
equivariant map $S/C \to \tilde{\SL}_2(\RR)/H$. Furthermore, we may push forward the invariant measure on $S/C$ to $\tilde{\SL}_2(\RR)/H$. By the second part of Lemma \ref{lem:SL2_dec}, we conclude that the projection of $C_\circ$, which is a normal subgroup in $H$, must be trivial. This implies that $C_{\circ}$ is trivial.
\end{proof}

\begin{prop}\label{prop:pHisdiscrete}
If $\p(H_{\circ})$ is solvable, then $\p(H)$ is discrete in $S$. 
\end{prop}
\begin{proof}
Since $H$ is a uniform subgroup of $L$, the closure $C$ of 
$\p(H)$ is a uniform subgroup in $S$. Note that 
$C$ contains the closed 
subgroup $\p(H_{\circ})$ as a normal subgroup.
Moreover, $S/C$ has a finite $S$-invariant measure.
So Lemma \ref{lem:SmodC} applies and shows that $C$ is
discrete. Hence, the subgroup $\p(H) \subset C$ is discrete.
\end{proof}

We shall also need: 

\begin{lem}\label{lem:algebraiclemma2}
Let $\frl$ be a Lie algebra with Levi decomposition
$\frl=\frs\ltimes\frr$, where $\frr$ is the solvable radical of
$\frl$ and $\frs$ a Levi subalgebra.
Furthermore, let $\frn\subset\frr$ denote the nilradical of $\frr$.
For an ideal $\frs_1$ in $\frs$, let
$\frb$ denote the ideal in $\frn$ generated by $[\frs_1,\frn]$.
Then $\frb$ is an ideal in $\frl$.
\end{lem}
\begin{proof}
First, recall that $[\frs_1,\frr]=[\frs_1,\frn]$, since $\frs_1$ acts reductively on $\frr$ and it acts trivially on $\frr/\frn$  
(see Remark \ref{rem_der}).
Let $X=[S_1,N]$, where $S_1\in\frs_1$,
$N\in\frn$, and let $D\in\frr$.
Then there exists $N_1\in\frn$ such that $[D, S_1] = [N_1, S_1]$. Therefore, 
\begin{align*}
[D,X] &= [D,[S_1,N]] = -[N,[D,S_1]] - [S_1,[N,D]] \\
&= -\underbrace{[N,[N_1,S_1]]}_{\in\frb}-\underbrace{[S_1,[N,D]]}_{\in[\frs_1,\frn]\subset\frb} \; .
\end{align*}
Thus $[\frr,[\frs_1,\frn]]\subset\frb$. Taking into account that
$\frb$ is an ideal in $\frn$,
we deduce that  $[\frr,\frb]\subset\frb$.
For all $S\in\frs$, $[S,\frs_1]\subset\frs_1$.
Hence
\[
[S,[S_1,N]]
=
-[S_1,[N,S]]-[N,[S_1,S]]\in[\frs_1,\frn].
\]
This again implies $[\frs,\frb]\subset\frb$.
Therefore, $\frb$ is an ideal in $\frl$.
\end{proof}

For the proof of Theorem \ref{thm:asph_invarvol}, let us first assume that $\p(H_{\circ})$ is solvable.
Thus Proposition \ref{prop:pHisdiscrete} implies that $\p(H)$ is a uniform lattice in $S$. In particular, $H_{\circ} \leq R$ and $H \cap R$ is a uniform subgroup in $R$. 

In the general case, if $\p(H_{\circ})$ projects onto a simple
factor $S_1$ of $S$, we can remove the factor $S_{1}$ from $L$.
The remaining subgroup of $L$ still acts transitively on $M$.
Iterating this procedure, we arrive at a subgroup $L'$ of $L$,
such that $\p(H_{\circ} \cap L')$ is solvable and $L'$ acts
transitively on $M$. Note that $R$ is contained in $L'$ by construction. By the first part of the proof, we see that $H \cap R$ is a uniform subgroup in $R$.

Let $N$ be the nilradical of $R$. Since $H \cap R$ is uniform in
$R$, $H \cap N$ is a uniform subgroup in $N$.
This shows that $N \cap H_\circ$ is a normal subgroup of $N$
(as was already known to  Malcev \cite{malcev}).

Let $S_1$ be a Levi subgroup of $H_\circ$. As follows from the 
above construction, $S_1$ is (conjugate to) a factor of $S$.  

Since $H$ normalizes the lattice subgroup
$(H \cap N)/(H_\circ \cap N)$, which does not admit any connected
group of automorphisms, it follows that, for all $h\in H_{\circ}$,
\[
\Ad(h)|_{N/(H_\circ \cap N)} = \id \; .
\]
In particular, this applies to all $h \in S_1 \subset H_{\circ}$.
Therefore, $[\frs_1, \frn]$ is contained in $\frh \cap \frn$,  
where $\frs_1, \frn, \frh$ denote the Lie algebras of $S_1, N$,
and $H$, respectively.

Let $\frb$ be the ideal in $\frn$ generated by $[\frs_1, \frn]$.
Since $\frh \cap \frn$ is an ideal in $\frn$, evidently,
$\frb \subset \frh \cap \frn$.
By Lemma \ref{lem:algebraiclemma2}, $\frb$ is an ideal in the
Lie algebra $\frl$ of $L$.
Since $\frb\subset\frh$ and $L$ acts almost effectively, we must have
$\frb=\zsp$.
Let $\frr$ be the Lie algebra of $R$. Since $\frs_1$ acts reductively
on $\frr$ and it acts trivially on $\frr/\frn$, 
\[
[\frs_1,\frr]=[\frs_1,\frn] \subset \frb = \zsp \; .
\]
So the subgroup $S_1$ of $H_\circ$ centralizes $R$ and is
therefore also normal in $L$.
Again, since $L$ acts almost effectively, we must have
$S_1 = \one$. In conclusion, we 
have that $H_{\circ }$ is contained in $R$. In particular, 
$H_{\circ }$ is solvable and by Proposition \ref{prop:pHisdiscrete},
$\p(H)$ is discrete in $S$.
This shows Theorem \ref{thm:asph_invarvol}. 
\section{Isometric presentations}
\label{sec_rigidity}

Let $\mathcal{P}=(G,\g_G,\Gamma,\phi)$ be a presentation for a
compact pseudo-Rie\-mannian mani\-fold $M$ by a Lie group with bi-invariant metric, and let $x_{0} = \phi(e\,  \Gamma)$ be the base point.  We note that, via $\phi$, the group $G$ acts on $M$ by isometries. Then a change of base point in $M$ from $x_{0}$ to $a\cdot x_{0}$, $a\in G$, corresponds to an isometry of presentations 
for $M$:

\begin{lem}\label{lem:changebase}

Let $a\in G$ and $\Gamma^a=a\Gamma a^{-1}$.
Then there exist a presentation
$\mathcal{P}^a=(G,\g_G,\Gamma^a,\phi^a)$
for $M$ which is isometric to $\mathcal{P}$ and
satisfies $\phi^a(e \, \Gamma^a) = a \cdot x_{0}$.
\end{lem}
\begin{proof}
Let $\lambda_{a}: M \to  M$, $x \mapsto a\cdot x$ be the
isometry of $M$ which belongs to $a$ with respect to
$\mathcal{P}$.
Consider the isomorphism
$\Psi_{a}:G\to G$, $g\mapsto aga^{-1}$.
Then clearly $\Psi_{a}(\Gamma)=\Gamma^a$, and since $\g_G$ is
bi-invariant, $\Psi_{a}:G\to G$ is an isometry for $\g_G$. 
Define $\phi^a = \lambda_{a} \phi \, \bar \Psi_{a}^{-1}: G/\, \Gamma^{a} \to M$. It follows that $\phi^a$ is an isometry with the required property, and $\Psi_{a}$ defines an isometry of presentations $\mathcal{P} \to \mathcal{P}^a$. 
\end{proof}


Let $\psi:M_1\to M_2$ be an isometry, $\psi(x_{1}) =x_{2}$, where
$x_{i} = \phi_i( e \, \Gamma_{i}) \in M_{i}$ are the base points.
Then there is an associated isomorphism of groups
\[
J_{\psi}:  \Iso(M_{1}) \to \Iso(M_{2})  \, ,\quad 
\sigma \mapsto \psi \sigma \psi^{-1}
\]
which maps $\Iso(M_{1})_{\circ}$ to $\Iso(M_{2})_{\circ}$.  
Since the simply connected groups $G_{i}$ act almost freely and by isometries on $M_{i}$, the natural maps 
\[
G_{i} \to \Iso(M_{i})_{\circ}
\]
have discrete kernels. Indeed, by Corollary \ref{mcor_isoM},
these maps are surjective, that is, they are covering homomorphisms. Let $$ \Psi: G_{1} \to G_{2}$$ be the unique lift of $J_{\psi}$  to an isomorphism of the universal covering groups $G_{i}$.  Then, clearly, $\Psi(\Gamma_{1}) = \Gamma_{2}$, and there is a map 
\[
\tilde{\Psi}: M_{1} \to M_{2}
\]
induced by $\Psi$. 
Moreover, for $g\in G_1$, we have
\begin{align*}
\tilde{\Psi}(g \cdot x_{1})
& =\Psi(g)\cdot x_{2} 
 =J_{\psi}(\lambda_g)(x_{2})
=\psi\lambda_g\psi^{-1}(x_{2})
=\psi\lambda_g\psi^{-1}(\psi(x_{1})) \\
&=\psi(g \cdot {x_{1}}).
\end{align*}
Hence, $\tilde{\Psi} = \psi$. 
In particular, since $\psi$ is an isometry, the isomorphism $\Psi: G_{1} \to G_{2}$
is an isometry of the pulled-back metrics $\g_{G_{i}}$.
Thus $\Psi$ defines an isometry of pre\-sen\-tations
$\mathcal{P}_{1} \to \mathcal{P}_{2}$, which induces $\psi$.
This proves the first part of Corollary \ref{mcor_rigidity}. 

Now let $\mathcal{P}_{i}$ be two presentations of $M$. After a change of base point in $M$ and
a corresponding isometric change of the presentation
$\mathcal{P}_{2}$ (as in Lemma \ref{lem:changebase}),
we can assume that $\mathcal{P}_{1}$ and $\mathcal{P}_{2}$
have the same base-point $x_{0}$. 
According to the first part of the proof, 
the identity of $M$, $\psi= \id_{M}$, lifts to an isometry $\mathcal{P}_{1} \to \mathcal{P}_{2}$.
This finishes the proof of  Corollary \ref{mcor_rigidity}.

%
%
%
%
%
%













\begin{thebibliography}{99}

\bibitem{AS1} S. Adams, G. Stuck,
{\em The isometry group of a compact Lorentz manifold I},
Invent. Math. 129, 1997, 239-261

\bibitem{an} J. An,
{\em Rigid geometric structures, isometric actions, and algebraic
quotients},
Geom. Dedicata 157, 2012, 153-185



\bibitem{baues2} O. Baues,
{\em Prehomogeneous affine representations and flat pseudo-Riemannian manifolds}, 
in `Handbook of Pseudo-Riemannian Geometry and Supersymmetry', EMS IRMA Lect. Math. Theor. Phys. {16}, 2010, 731-817


\bibitem{BK} O. Baues, B. Klopsch,
{\em Deformations and rigidity of lattices in solvable Lie groups},
J. Topo\-logy 6, 2013 (4), 823-856



\bibitem{borel0} A. Borel,
{\em Density properties for certain subgroups of semi-simple groups without compact components},
Ann. Math. 72, 1960 (1), 179-188

\bibitem{borel} A. Borel,
{\em Linear Algebraic Groups}, second edition,
Springer, 1991

%


\bibitem{CR} P.E. Conner, F. Raymond,
{\em Actions of compact Lie groups on aspherical manifolds},
in `Topology of Manifolds (Proc. Univ. of Georgia Conf., 1969)', Markham, 1970, 227-264

\bibitem{DAG} G. D'Ambra, M. Gromov,
{\em Lectures on Transformation Groups: Geometry and Dynamics},
Surv. Differ. Geom. 1, 1991, 19-111





\bibitem{gromov} M. Gromov,
{\em Rigid transformation groups},
in `G\'eom\'etrie diff\'erentielle',
Travaux en Cours 33, Hermann, 1988, 65-139




\bibitem{johnson} R.W. Johnson,
{\em Presentations of Solvmanifolds},
Amer. J. Math. 94, 1972 (1), 82-102


%

%


\bibitem{jacobson} N. Jacobson,
{\em Lie Algebras},
Wiley, 1962

\bibitem{lang} S. Lang,
{\em Algebra}, third edition,
Springer, 2002

\bibitem{malcev} A.I. Malcev,
{\em On a class of homogeneous spaces},
Amer. Math. Soc. Translation 39, 1951



\bibitem{MR} A. Medina, P. Revoy,
{\em Alg\`ebres de Lie et produit scalaire invariant},
Ann. scient. \'Ec. Norm. Sup. 18, 1985 (3), 553-561

\bibitem{MR2} A. Medina, P. Revoy,
{\em Les groupes oscillateurs et leurs resaux},
Manuscripta Math. 52, 1985, 81-95


\bibitem{mostow0} G.D. Mostow,
{\em Factor Spaces of Solvable Lie Groups},
Ann. Math. 60, 1954 (1), 1-27


\bibitem{oneill} B. O'Neill,
{\em Semi-Riemannian Geometry},
Academic Press, 1983

\bibitem{quiroga} R. Quiroga-Barranco,
{\em Isometric actions of simple Lie groups on pseudoRiemannian
manifolds},
Ann. Math. 164, 2006, 941-969

\bibitem{raghunathan} M.S. Raghunathan,
{\em Discrete Subgroups of Lie Groups},
Springer, 1972

\bibitem{saito} M. Saito,
{\em Sous-groupes discrets des groupes resolubles},
Amer. J. Math. 83, 1961 (2), 369-392







\bibitem{zeghib} A. Zeghib,
{\em Sur les espaces-temps homog\`enes},
Geometry and Topology Monographs 1: The Epstein Birthday Schrift, 1998, 551-576


\bibitem{zimmer} R.J. Zimmer,
{\em Ergodic Theory and Semisimple Groups},
Birkh\"auser, 1984



\bibitem{zimmer4} R.J. Zimmer,
{\em On the automorphism group of a compact Lorentz manifold and other geometric manifolds},
Invent. Math. 83, 1986, 411-424


\bibitem{ZB} P.B. Zwart, W.M. Boothby,
{\em On compact homogeneous symplectic manifolds},
Ann. Inst. Fourier 30, 1980 (1), 129-157

\end{thebibliography}
\end{document}